\documentclass[12pt]{article}
\usepackage{lineno,hyperref}
\modulolinenumbers[5]
\usepackage{color}

\usepackage{graphics,amsmath,amssymb}
\usepackage{amsthm}
\usepackage{amsfonts}
\usepackage{latexsym}
\usepackage{epsf}

\newcommand{\fstirling}[2]{\genfrac[]{0pt}{}{#1}{#2}}
\newcommand{\sstirling}[2]{\genfrac\{\}{0pt}{}{#1}{#2}}

\theoremstyle{plain}
\newtheorem{theorem}{Theorem}
\newtheorem{corollary}[theorem]{Corollary}

\theoremstyle{definition}

\theoremstyle{remark}
\newtheorem{remark}[theorem]{Remark}
\newtheorem{problem}[theorem]{Problem}

\begin{document}

\begin{center}
\vskip 1cm{\LARGE\bf Some Theorems and Applications of the\\
\vskip .1in
$(q,r)$-Whitney Numbers}
\vskip 1cm
\large    
Mahid M. Mangontarum\\
Department of Mathematics\\
Mindanao State University---Main Campus\\
Marawi City 9700\\
Philippines \\
\href{mailto:mmangontarum@yahoo.com}{\tt mmangontarum@yahoo.com} \\
\href{mailto:mangontarum.mahid@msumain.edu.ph}{\tt mangontarum.mahid@msumain.edu.ph} \\
\end{center}

\vskip .2 in

\begin{abstract}
The $(q,r)$-Whitney numbers were recently defined in terms of the
$q$-Boson operators, and several combinatorial properties which appear to
be $q$-analogues of similar properties were studied. In
this paper, we obtain elementary and complete symmetric polynomial
forms for the $(q,r)$-Whitney numbers, and give combinatorial
interpretations in the context of $A$-tableaux. We also obtain
convolution-type identities using the combinatorics of $A$-tableaux.
Lastly, we present applications and theorems related to discrete
$q$-distributions.

\medskip


\medskip


\end{abstract}

\section{Introduction}
\label{sec:Introduction}
In a recent paper, the author and Katriel \cite{Mah2} introduced a new
approach to generate $q$-analogues of Stirling and Whitney-type numbers. 
In this paper, the $(q,r)$-Whitney numbers of the first and second kinds were defined as coefficients in
\begin{equation}
m^n(a^\dagger)^n a^n=\sum_{k=0}^{n}w_{m,r,q}(n,k)(ma^\dagger a+r)^k\label{qw1}
\end{equation}
and
\begin{equation}
(ma^\dagger a+r)^n=\sum_{k=0}^{n}m^kW_{m,r,q}(n,k)(a^\dagger)^k a^k,\label{qw2}
\end{equation}
respectively (cf.\ \cite{Mah2}), by using as framework, the $q$-Boson operators $a^{\dagger}$ and $a$ of Arik and Coon \cite{Arik} which satisfy the commutation relation
\begin{equation}
[a,a^\dagger]_q\equiv aa^\dagger-qa^\dagger a=1.\label{qbos}
\end{equation}
By convention, $w_{m,r,q}(0,0)=W_{m,r,q}(0,0)=1$ and $w_{m,r,q}(n,k)=W_{m,r,q}(n,k)=0$ for $k>n$ and for $k<0$. Several combinatorial properties were already established, including the following triangular recurrence relations \cite[Theorem\ 6]{Mah2}:
\begin{equation}
w_{m,r,q}(n+1,k)=q^{-n}\Big(w_{m,r,q}(n,k-1)-(m[n]_q+r)w_{m,r,q}(n,k)\Big), \label{identity4}
\end{equation}
with $[n]_q=\frac{q^n-1}{q-1}$, the $q$-integer, and
\begin{equation}
W_{m,r,q}(n+1,k)=q^{k-1}W_{m,r,q}(n,k-1)+(m[k]_q+r)W_{m,r,q}(n,k).\label{identity5}
\end{equation}
From here, one readily obtains
\begin{equation}
w_{m,r,q}(n,0)=(-1)^nq^{-\binom{n}{2}}\prod_{i=0}^{n-1}(m[i]_q+r),\label{w0}
\end{equation}
\begin{equation}
w_{m,r,q}(n,n)=q^{-\binom{n}{2}},\label{wn}
\end{equation}
\begin{equation}
W_{m,r,q}(n,0)=r^n,\label{W0}
\end{equation}
and
\begin{equation}
W_{m,r,q}(n,n)=q^{\binom{n}{2}}.\label{Wn}
\end{equation}
The identities presented in Eqs.~\eqref{identity4} and \eqref{identity5} can be used as tools to obtain further combinatorial identities for $w_{m,r,q}(n,k)$ and $W_{m,r,q}(n,k)$. For instance, with the aid of these recurrence relations, the vertical recurrence relations
\begin{equation}
w_{m,r,q}(n+1,k+1)=\sum_{j=k}^{n}(-1)^{n-j}q^{\binom{j}{2}-\binom{n+1}{2}}w_{m,r,q}(j,k)\prod_{i=j+1}^n(m[i]_q+r),\label{vertrec1}
\end{equation}
with $\prod_{i=j+1}^n(m[i]_q+r)=1$ when $j=n$, and
\begin{equation}
W_{m,r,q}(n+1,k+1)=q^k\sum_{j=k}^{n}(m[k+1]_q+r)^{n-j}W_{m,r,q}(j,k),\label{vertrec2}
\end{equation}
can be proved by induction, as well as the rational generating function of the $(q,r)$-Whitney numbers of the second kind given by
\begin{equation}
\sum_{n=k}^{\infty}W_{m,r,q}(n,k)t^n=\frac{q^{\binom{k}{2}}t^k}{\prod_{i=0}^k\left(1-(m[i]_q+r)t\right)}.\label{ratgenf}
\end{equation}
On the other hand, the horizontal recurrence relations 
\begin{equation}
w_{m,r,q}(n,k)=q^n\sum_{j=0}^{n-k}(m[n]_q+r)^jw_{m,r,q}(n+1,k+j+1)\label{horrec1}
\end{equation}
and
\begin{equation}
W_{m,r,q}(n,k)=\sum_{j=0}^{n-k}(-1)^{j}q^{\binom{k}{2}-\binom{k+j+1}{2}}\frac{\prod_{i=0}^{k+j}(m[i]_q+r)}{\prod_{i=0}^{k}(m[i]_q+r)}W_{m,r,q}(n+1,k+j+1)\label{horrec2}
\end{equation}
can be verified by evaluating the right-hand sides using 
Eqs.~\eqref{identity4} and \eqref{identity5}.
Before proceeding, we note that Eqs.~\eqref{vertrec1} and \eqref{vertrec2} follow a behaviour similar to that of the Chu-Shi-Chieh's identity (see \cite{Chen}) for the classical binomial coefficients given by
\begin{equation*}
\binom{n+1}{k+1}=\binom{k}{k}+\binom{k+1}{k}+\cdots+\binom{n}{k},
\end{equation*}
while Eqs.~\eqref{horrec1} and \eqref{horrec2} are analogous with
\begin{equation*}
\binom{n}{k}=\binom{n+1}{k+1}-\binom{n+1}{k+2}+\cdots+(-1)^{n-k}\binom{n+1}{n+1},
\end{equation*}
another known identity for the classical binomial coefficients.

The purpose of this paper is to express the $(q,r)$-Whitney numbers of both kinds in symmetric polynomial forms. This proves to be useful in establishing combinatorial interpretations in terms of $A$-tableaux. In return, remarkable convolution-type identities are obtained and several other interesting theorems are also presented.

\section{Explicit formulas in symmetric polynomial forms}

\subsection{$(q,r)$-Whitney numbers of the first kind}

Expanding the falling factorial $(x)_{n}=x(x-1)\cdots(x-n+1)$ in powers of $x$, we obtain
$$(x)_{n}=\sum_{k=0}^n (-1)^{n-k}x^k \sum_{1\leq i_1<i_2<\cdots<i_{n-k}\leq n-1}\prod_{j=1}^{n-k} i_j \, ,$$
which yield the well-known expression for the Stirling numbers of the first kind in terms of elementary symmetric functions. This relation can be generalized to the $q$-Stirling numbers as follows:
\begin{eqnarray*}
[x]_q[x-1]_q\cdots[x-n+1]_q &=& [x]_q([x]_q+q^x)([x]_q+q^x[2]_q)\cdots([x]_q+q^x[n-1]_q) \\
&=& \sum_{k=0}^n [x]_q^k\cdot q^{x(n-k)}\sum_{1\leq i_1<i_2<\cdots<i_{n-k}\leq n-1}\prod_{j=1}^{n-k}[i_j]_q \, .
\end{eqnarray*}

To further generalize this procedure to the $(q,r)$-Whitney numbers of the first kind, recall that application of both sides of the defining relation in 
Eq.~\eqref{qw1} on the $q$-boson number state $\left.|\ell\right\rangle$ gives
$$m^n[\ell]_q[\ell-1]_q\cdots[\ell-n+1]_q=\sum_{k=0}^n w_{m,r,q}(n,k)\Big(m[\ell]_q+r\Big)^k \, .$$
Since both sides of this relation are finite polynomials in $\ell$, and since the relation is valid for all integer $\ell$, it remains valid when $\ell$ is replaced by the real number $x$, i.e.,
\begin{equation}
m^n[x]_q[x-1]_q\cdots[x-n+1]_q=\sum_{k=0}^n w_{m,r,q}(n,k)\Big(m[x]_q+r\Big)^k \, .\label{wh1}
\end{equation}
Now, defining $y=[x]_q+\alpha$, where $\alpha=\frac{r}{m}$, we note that $[x-i]_q=q^{-i}(y-\alpha-[i]_q)$. Hence,
\begin{equation}
m^n[x]_q[x-1]_q\cdots[x-n+1]_q=\sum_{k=0}^n (m[x]_q+r)^k q^{-\binom{n}{2}}(-1)^{n-k}\sum_{0\leq i_1<i_2<\cdots<i_{n-k}\leq n-1}\prod_{j=1}^{n-k}(m[i_j]_q+r) \, .\label{wh2}
\end{equation}
The identity in the next theorem is obtained by comparing the right-hand-sides of Eqs.~\eqref{wh1} and \eqref{wh2}.
\begin{theorem}\label{t1}
The $(q,r)$-Whitney numbers of the first kind satisfy the following explicit form
\begin{equation}
w_{m,r,q}(n,k)=(-1)^{n-k}q^{-\binom{n}{2}}\sum_{0\leq i_1<i_2<\cdots<i_{n-k}\leq n-1}\prod_{j=1}^{n-k}(r+[i_j]_qm).\label{t1.1}
\end{equation}
\end{theorem}
\begin{remark}
The sum within this theorem is the symmetric polynomial of degree $n-k$ in the $n$ variables $\{(r+[i]_qm);i=0,1,\ldots, n-1\}$. For $r=0$ all the terms with $i_1=0$ vanish so the summation starts at 1, which is consistent with the expressions presented above for the Stirling and $q$-Stirling numbers of the first kind.
\end{remark}
The above theorem can also be proved by induction as follows:
\begin{proof}[Alternative proof of Theorem \ref{t1}]
The theorem readily yields $w_{m,r,q}(0,0)=1$.
Making the induction hypothesis that the theorem is true up to $n$, for all $k=0,1,\ldots, n$, we prove it
for $n+1$ and $k=0,1,\ldots,n$, via the recurrence relation \eqref{identity4}. Thus,
\begin{eqnarray*}
w_{m,r,q}(n+1,k) &=& q^{-n}\Big( (-1)^{n+1-k}q^{-\binom{n}{2}}\sum_{0\leq i_1<i_2<\cdots<i_{n+1-k}\leq n-1}\prod_{j=1}^{n+1-k}(r+[i_j]_qm) \\
& & -(m[n]_q+r)(-1)^{n-k}q^{-\binom{n}{2}}\sum_{0\leq i_1<i_2<\cdots<i_{n-k}\leq n-1}\prod_{j=1}^{n-k}(r+[i_j]_qm) \Big) \\
&=& q^{-\binom{n+1}{2}}(-1)^{n+1-k}q^{-\binom{n}{2}}\Big(  \sum_{0\leq i_1<i_2<\cdots<i_{n+1-k}\leq n-1}\prod_{j=1}^{n+1-k}(r+[i_j]_qm) \\
& & +(m[n]_q+r)\sum_{0\leq i_1<i_2<\cdots<i_{n-k}\leq n-1}\prod_{j=1}^{n-k}(r+[i_j]_qm) \Big)
\end{eqnarray*}
The first term within the large paretheses contains all products of $n+2-k$ distinct factors out of 
$\{ (r+[i]_qm);i=0,1,\ldots,n-1\}$, whereas the second term contains all products of $n+2-k$ distinct factors,
one of which is $(r+m[n]_q)$ and the others chosen out of $\{ (r+[i]_qm);i=0,1,\ldots,n-1\}$.
Together, these sums yield 
$$\sum_{0\leq i_1<i_2<\cdots<i_{n+1-k}\leq n}\prod_{j=0}^{n+1-k}(r+[i_j]_qm),$$
thus establishing the theorem for the range of indices specified above.
Finally, the theorem yields $w_{m,r,q}(n+1,n+1)=q^{-\binom{n+1}{2}}$, in agreement with \eqref{wn}.
\end{proof}

As $q\rightarrow 1$, the explicit formula \eqref{t1.1} reduces to an expression for the $r$-Whitney numbers of the first kind given by
\begin{equation}
w_{m,r}(n,k)=(-1)^{n-k}\sum_{0\leq i_1<i_2<\cdots<i_{n-k}\leq n-1}\prod_{j=1}^{n-k}(r+i_jm).
\end{equation}
An equivalent of this identity was reported by Mangontarum et al. \cite[Theorem\ 6]{Mah3}. For $m=1$ and $r=0$, \eqref{t1.1} reduces to an explicit formula for a $q$-analogue of the Stirling numbers of the first kind, viz,
\begin{equation}
\fstirling{n}{k}_q=(-1)^{n-k}q^{-\binom{n}{2}}\sum_{0\leq i_1<i_2<\cdots<i_{n-k}\leq n-1}\prod_{j=1}^{n-k}[i_j]_q,\label{t1.2}
\end{equation}
where $\fstirling{n}{k}_q$ denote the $q$-Stirling numbers of the first kind defined by
\begin{equation}
[x]_{q,n}=\sum_{k=1}^n(-1)^{n-k}\fstirling{n}{k}_q[x]_q^k,
\end{equation}
$[x]_{q,n}=[x]_q[x-1]_q[x-2]_q\cdots[x-n+1]_q$ (cf.\ \cite{Car3}). For any given set of $n-k$ integers that satisfy $1<i_2<\cdots<i_{n-k}<n-1$, let
$$\left\{\ell_1, \ell_2, \ldots, \ell_k\right\}\equiv\left\{1,2,3,\ldots,n-1\right\}-\left\{i_1,i_2,\ldots,i_{n-k}\right\}$$
be the complement with respect to $\{1,2,3,\ldots,n-1\}$. It follows that
\begin{equation}
\prod_{j=0}^{n-k}[i_j]_q=\frac{[n-1]_q!}{\prod_{j=0}^k[\ell_j]_q}.
\end{equation}
This allows \eqref{t1.2} to be written in the form
\begin{equation}
\fstirling{n}{k}_q=q^{-\binom{n}{2}}[n-1]_q!\sum_{0\leq i_1<i_2<\cdots<i_{n-k}\leq n-1}\frac{1}{\prod_{j=0}^k[\ell_j]_q}.\label{t1.3}
\end{equation}
As $q\rightarrow 1$, one recovers from \eqref{t1.2} Comtet's \cite{Comt} identity given by
\begin{equation}
\fstirling{n}{k}=(-1)^{n-k}\sum_{0\leq i_1<i_2<\cdots<i_{n-k}\leq n-1}\prod_{j=1}^{n-k}i_j,
\end{equation}
while \eqref{t1.3} yields Adamchik's \cite{Adam} identity for the Stirling numbers of the first kind given by
\begin{equation}
\fstirling{n}{k}=(n-1)!\sum_{0\leq i_1<i_2<\cdots<i_{n-k}\leq n-1}\frac{1}{\prod_{j=0}^k\ell_j}.
\end{equation}

\subsection{$(q,r)$-Whitney numbers of the second kind}

\begin{theorem}
The $(q,r)$-Whitney numbers of the second kind satisfy the following explicit form:
\begin{equation}
W_{m,r,q}(n,k)=q^{\binom{k}{2}}\sum_{c_0+c_1+\cdots+c_k=n-k}\prod_{j=0}^{k}(m[j]_q+r)^{c_j}, \label{expl2}
\end{equation}
where $c_0, c_1, \ldots, c_k$ are non-negative integers.
\end{theorem}
\begin{proof}
We proceed by induction over $n$. First, we note that the theorem is satisfied when $n=k=0$. That is, $W_{m,r,q}(0,0)=1$.
Making the induction hypothesis that the theorem holds up to $n$ (for all $k=0,1,\ldots,n$) we show,
using the recurrence relation \eqref{identity5}, that it holds for $n+1$ and $k=0,1,\ldots,n$. Thus,
\begin{eqnarray*}
W_{m,r,q}(n+1,k) &=& q^{k-1}q^{\binom{k-1}{2}}\sum_{c_0+c_1+\cdots+c_{k-1}=n+1-k}\prod_{j=0}^{k-1}(m[j]_q+r)^{c_j}  \\
    & & \;\;\;\;\;\; +(m[k]_q+r)q^{\binom{k}{2}}\sum_{c_0+c_1+\cdots+c_k=n-k}\prod_{j=0}^k(m[j]_q+r)^{c_j} \\
    &=& q^{\binom{k}{2}}\Big(\sum_{c_0+c_1+\cdots+c_{k-1}=n+1-k}\prod_{j=0}^{k-1}(m[j]_q+r)^{c_j}  \\
    & & \;\;\;\;\;      +(m[k]_q+r)\sum_{c_0+c_1+\cdots+c_k=n-k}\prod_{j=0}^k(m[j]_q+r)^{c_j} \Big)
\end{eqnarray*}
Now, the first term within the big paretheses is a sum of products of $n+1-k$ factors, non of which contains $(m[k]_q+r)$.
The second term is again a sum of $n+1-k$ factors, each one of which containing $(m[k]_q+r)$ at least once.
Thus, 
$$W_{m,r,q}(n+1,k)=q^{\binom{k}{2}}\sum_{c_0+c_1+\cdots+c_k=n+1-k}\prod_{j=0}^k(m[j]_q+r)^{c_j}.$$
To complete the proof we need to show that the theorem holds for $n+1$ and $k=n+1$. For this case the theorem yields
$W_{m,r,q}(n+1,n+1)=q^{\binom{n+1}{2}}$, which is in agreement with \eqref{Wn}.
\end{proof}
Apart from $q^{\binom{k}{2}}$, \eqref{expl2} is a homogeneous complete symmetric polynomial of degree $n-k$ in the variables $\{(r+[j]_qm);j=0,1,2,\ldots,k\}$. As $q\rightarrow 1$, we obtain
\begin{equation}
W_{m,r}(n,k)=\sum_{c_0+c_1+\cdots+c_k=n-k}\prod_{j=0}^{k}(r+mj)^{c_j},
\end{equation}
and for $r=0$, \eqref{expl2} reduces to an expression for the $q$-Stirling numbers of the second kind, viz,
\begin{equation}
\sstirling{n}{k}_q=q^{\binom{k}{2}}\sum_{c_0+c_1+\cdots+c_k=n-k}[1]_q^{i_1}[2]_q^{i_2}\cdots[k]_q^{i_k}.\label{qS2}
\end{equation}
The $q$-Stirling numbers of the second kind were originally defined as
\begin{equation}
[x]_q^n=\sum_{k=1}^n\sstirling{n}{k}_q[x]_{q,k}
\end{equation}
(cf.\ \cite{Car3}). Moreover, when $q\rightarrow 1$, Eq.~\eqref{qS2} yields an expression for the classical Stirling numbers of the second kind reported by Comtet \cite{Comt}. 

Notice that from the inner product
\begin{equation}
\prod_{j=0}^{k}(m[j]_q+r)^{c_j}=(m[0]_q+r)^{c_0}(m[1]_q+r)^{c_1}(m[2]_q+r)^{c_2}\cdots(m[k]_q+r)^{c_k}
\end{equation}
in the explicit formula in \eqref{expl2}, we observe that there are exactly $n-k$ factors of $(m[j]_q+r)$ which is repeated $c_j$ times for each $j$. From here, we write 
\begin{equation*}
(m[0]_q+r)^{c_0}=(m[j_1]_q+r)(m[j_2]_q+r)\cdots(m[j_{c_0}]_q+r),
\end{equation*}
where $j_i=0$, $i=1,2,\ldots,c_0$;
\begin{equation*}
(m[1]_q+r)^{c_1}=(m[j_{c_0+1}]_q+r)(m[j_{c_0+2}]_q+r)\cdots(m[j_{c_0+c_1}]_q+r),
\end{equation*}
where $j_{c_0+i}=1$, $i=1,2,\ldots,c_1$; 
\begin{equation*}
(m[2]_q+r)^{c_2}=(m[j_{c_0+c_1+1}]_q+r)(m[j_{c_0++c_1+2}]_q+r)\cdots(m[j_{c_0+c_1+c_2}]_q+r),
\end{equation*}
where $j_{c_0+c_1+i}=2$, $i=1,2,\ldots,c_2$ and so on until
\begin{equation*}
(m[k]_q+r)^{c_k}=(m[j_{c_0+c_1+\cdots+c_{k-1}+1}]_q+r)(m[j_{c_0+c_1+\cdots+c_{k-1}+2}]_q+r)\cdots(m[j_{c_0+c_1+\cdots+c_k}]_q+r),
\end{equation*}
where $j_{c_0+c_1+\cdots+c_{k-1}+i}=k$, $i=1,2,\ldots,c_k$ and $c_0+c_1+c_2+\cdots+c_{k-1}+c_k=n-k$. Thus, $0\leq j_1\leq j_2\leq \cdots\leq j_{n-k}\leq k$ and we have
\begin{equation}
W_{m,r,q}(n,k)=q^{\binom{k}{2}}\sum_{0\leq j_1\leq j_2\leq \cdots\leq j_{n-k}\leq k}\prod_{i=1}^{n-k}(m[j_i]_q+r).
\end{equation}
We formally state this result in the next theorem.
\begin{theorem}
The $(q,r)$-Whitney numbers of the second kind satisfy the following explicit form:
\begin{equation}
W_{m,r,q}(n,k)=q^{\binom{k}{2}}\sum_{0\leq j_1\leq j_2\leq \cdots\leq j_{n-k}\leq k}\prod_{i=1}^{n-k}(m[j_i]_q+r).\label{expl3}
\end{equation}
\end{theorem}
Notice that when $q\rightarrow 1$, we obtain an identity similar to the result obtained by Mangontarum et al. \cite[Theorem\ 11]{Mah3}.

\section{On the context of $A$-tableaux}

De Medicis and Leroux \cite{Med} defined a 0-1 tableau to be a pair $\varphi=(\lambda,f)$, where $\lambda=(\lambda_1\geq\lambda_2\geq\cdots\geq\lambda_k)$ is a partition of an integer $m$ and $f=(f_{ij})_{1\leq j\leq \lambda_i}$ is a ``filling" of the cells of the corresponding Ferrers diagram of shape $\lambda$ with 0's and 1's such that exactly one 1 in each column. For instance, 
Figure~\ref{fig1} represents the 0-1 tableau $\varphi=(\lambda,f)$, where $\lambda=(8,7,5,4,1)$ with 
$$f_{14}=f_{16}=f_{18}=f_{22}=f_{25}=f_{27}=f_{33}=f_{41}=1$$
and $f_{ij}=0$ elsewhere for $1\leq j\leq\lambda_i$.
\begin{figure}[htbp]
\begin{center}
     \begin{tabular}{|c|c|c|c|c|c|c|c|} \hline
         0 & 0 & 0 & 1 & 0 & 1 & 0 & 1      \\ \hline
         0 & 1 & 0 & 0 & 1 & 0 & 1        \\ \cline{1-7}
         0 & 0 & 1 & 0 & 0            \\ \cline{1-5}
         1 & 0 & 0 & 0             \\ \cline{1-4}
				 0                    \\ \cline{1-1}
     \end{tabular}
     \end{center}
     \caption{A 0-1 tableau $\varphi$}
     \label{0-1tableauexample}
     \label{fig1}
\end{figure}
In the same paper, an $A$-tableau is defined to be a list $\Phi$ of columns $c$ of a Ferrers diagram of a partition $\lambda$ (by decreasing order of length) such that the length $|c|$ is part of the sequence $A=(a_i)_{i\geq0}$, a strictly increasing sequences of non-negative integers. Combinatorial interpretations of Stirling-type numbers in terms of $A$-tableaux are already done in the past. Similar works can be seen in \cite{tina5,CorPJ,CorMon,Mah4,Med} and some of the references therein. In particular, Corcino and Montero \cite{CorMon} defined a $q$-analogue of the Rucinski-Voigt numbers (an equivalent of the $r$-Whitney numbers of the second kind) and then presented a combinatorial interpretation using the theory of $A$-tableaux. The same type of interpretation was obtained by Mangontarum et al. \cite{Mah4} for the case of the translated Whitney numbers (see \cite{Mangontarum3}) and their $q$-analogues. It is important to note that the $q$-analogues of these authors follow motivations which differ from that of the $(q,r)$-Whitney numbers. Furthermore, the numbers considered in the paper of Ram\'irez and Shattuck \cite{Ram} belong to $p,q$-analogues, a natural extension of $q$-analogues.

Now, we let $\omega$ be a function from the set of non-negative integers $N$ to a ring $K$, and suppose that $\Phi$ is an $A$-tableau with $r$ columns of length $|c|$. Also, it is known that $\Phi$ might contain a finite number of columns whose lengths are zero since $0\in A$ and if $\omega(0)\neq 0$ (cf.\ \cite{Med}). Before proceeding, we denote by $T^A(x,y)$ the set of $A$-tableaux with $A=\{0,1,2,\ldots,x\}$ and exactly $y$ columns (with some columns possibly of zero length), and by $T_d^A(x,y)$ the subset of $T^A(x,y)$ which contains all $A$-tableaux with columns of distinct lengths. The next theorem relates the $(q,r)$-Whitney numbers of both kinds to certain sets of $A$-tableaux.

\begin{theorem}\label{tableau1}
Let $\Omega:N\longrightarrow K$ and $\omega:N\longrightarrow K$ be functions from the set of non-negative integers $N$ to a ring $K$ (column weights according to length) defined by
$$\Omega(|c|)=m[|c|]_q+r$$
and
$$\omega(|c|)=m[|\bar{c}|]_q+r,$$
where $m$ and $r$ are complex numbers, $|c|$ is the length of column $c$ of an $A$-tableau in $T^A_d(n-1,n-k)$, and $|\bar{c}|$ is the length of column $c$ of an $A$-tableau in $T^A(k,n-k)$. Then
\begin{equation}
(-1)^{n-k}q^{\binom{n}{2}}w_{m,r,q}(n,k)=\sum_{\Phi\in T_d^A(n-1,n-k)}\prod_{c\in\Phi}\Omega(|c|)\label{tab1}
\end{equation}
and
\begin{equation}
q^{-\binom{k}{2}}W_{m,r,q}(n,k)=\sum_{\phi\in T^A(k,n-k)}\prod_{\bar{c}\in\phi}\omega(|\bar{c}|).\label{tab2}
\end{equation}
\end{theorem}
\begin{proof}
Let $\Phi\in T_d^A(n-1,n-k)$. This means that $\Phi$ has exactly $n-k$ columns, say $c_1,c_2,\ldots,c_{n-k}$ whose lengths are $j_1,j_2,\ldots,j_{n-k}$, respectively. Now, for each column $c_i\in\Phi$, $i=1,2,3,\ldots,n-k$, we have $|c_i|=j_i$ and
$$\Omega(|c_i|)=m[|j_i|]_q+r.$$
Thus,
\begin{eqnarray*}
\prod_{c\in\Phi}\Omega(|c|)&=&\prod_{i=1}^{n-k}\Omega(|c_i|)\\
&=&\prod_{i=1}^{n-k}(m[j_i]_q+r).
\end{eqnarray*}
Since $\Phi\in T_d^A(n-1,n-k)$, then
\begin{eqnarray*}
\sum_{\Phi\in T_d^A(n-1,n-k)}\prod_{c\in\Phi}\Omega(|c|)&=&\sum_{0\leq j_1<j_2<\cdots<j_{n-k}\leq n-1}\prod_{c\in\Phi}\Omega(|c|)\\
&=&\sum_{0\leq j_1<j_2<\cdots<j_{n-k}\leq n-1}\prod_{i=1}^{n-k}(m[j_i]_q+r)\\
&=&(-1)^{n-k}q^{\binom{n}{2}}w_{m,r,q}(n,k).
\end{eqnarray*}
The second result is obtained similarly.
\end{proof}

\subsection{Combinatorics of $A$-tableaux}

In the following theorem, we will demonstrate the simple combinatorics of $A$-tableaux. To start, note that Eqs.~\eqref{tab1} and \eqref{tab2} are equivalent to
\begin{equation}
(-1)^{n-k}q^{\binom{n}{2}}w_{m,r,q}(n,k)=\sum_{\Phi\in T_d^A(n-1,n-k)}\Omega_A(\Phi)\label{tab1.1}
\end{equation}
and
\begin{equation}
q^{-\binom{k}{2}}W_{m,r,q}(n,k)=\sum_{\phi\in T^A(k,n-k)}\omega_A(\phi),\label{tab2.1}
\end{equation}
respectively, where
\begin{equation}
\Omega_A(\Phi)=\prod_{c\in\Phi}\Omega(|c|)=\prod_{c\in\Phi}(m[|c|]_q+r),\ |c|\in\{0,1,2,\ldots,n-1\}\label{tab1.2}
\end{equation}
and
\begin{equation}
\omega_A(\phi)=\prod_{\bar{c}\in\phi}\omega(|\bar{c}|)=\prod_{\bar{c}\in\phi}(m[|c|]_q+r),\ |\bar{c}|\in\{0,1,2,\ldots,k\}.\label{tab2.2}
\end{equation}
\begin{theorem}
For nonnegative integers $n$ and $k$, and complex numbers $m$ and $r$, the following identities hold:
\begin{equation}
w_{m,r,q}(n,k)=\sum_{j=k}^n\binom{j}{k}(-r_2)^{j-k}w_{m,r_1,q}(n,j)\label{tab3}
\end{equation}
\begin{equation}
W_{m,r,q}(n,k)=\sum_{j=k}^n\binom{n}{j}r_2^{n-j}W_{m,r_1,q}(j,k),\label{tab4}
\end{equation}
where $r_1+r_2=r$.
\end{theorem}
\begin{proof}
Let $\Phi\in T_d^A(n-1)$. Substituting $j_i=|c|$ in Eq.~\eqref{tab1.2} gives
\begin{equation*}
\Omega_A(\Phi)=\prod_{i=1}^{n-k}(m[j_i]_q+r),
\end{equation*}
where $j_i\in\{0,1,2,\ldots,n-1\}$. Suppose that for some numbers $r_1$ and $r_2$, $r=r_1+r_2$. Then, with $\Omega^*(j_i)=m[j_i]_q+r_1$, we may write
\begin{eqnarray*}
\Omega_A(\Phi)&=&\prod_{i=1}^{n-k}\left[(m[j_i]_q+r_1)+r_2\right]\\
&=&\prod_{i=1}^{n-k}(\Omega^*(j_i)+r_2)\\
&=&\sum_{\ell=0}^{n-k}r_2^{n-k-\ell}\sum_{j_1\leq q_1<q_2<\cdots<q_{\ell}\leq j_{n-k}}\prod_{i=1}^{\ell}\Omega^{*}(q_i).
\end{eqnarray*}
Let $B_{\Phi}$ be the set of all $A$-tableaux corresponding to $\Phi$ such that for each $\psi\in B_{\Phi}$, one of the following is true:
\begin{description}
	\item[] $\psi$ has no column whose weight is $r_2$;
	\item[] $\psi$ has one column whose weight is $r_2$;
	\item[] $\psi$ has two columns whose weight are $r_2$;
	\item[] $\vdots$
	\item[] $\psi$ has $n-k$ columns whose weight are $r_2$.
\end{description}
Then,
\begin{equation*}
\Omega_A(\Phi)=\sum_{\psi\in B_{\Phi}}\Omega_A(\psi).
\end{equation*}
Now, if $\ell$ columns in $\psi$ have weights other than $r_2$, then
\begin{eqnarray*}
\Omega_A(\psi)&=&\prod_{c\in\psi}\Omega^*(|c|)\\
&=&r_2^{n-k\ell}\prod_{i=1}^{\ell}\Omega^*(q_i),
\end{eqnarray*}
where $q_1,q_2,q_3,\ldots,q_{\ell}\in\{j_1,j_2,j_3,\ldots,j_{n-k}\}$. Hence, 
Eq.~\eqref{tab1.1} may be written as
\begin{eqnarray*}
(-1)^{n-k}q^{\binom{n}{2}}w_{m,r,q}(n,k)&=&\sum_{\Phi\in T_d^A(n-1,n-k)}\Omega_A(\Phi)\\
&=&\sum_{\Phi\in T_d^A(n-1,n-k)}\sum_{\psi\in B_{\Phi}}\Omega_A(\psi).
\end{eqnarray*}
For each $\ell$, it is known that there correspond $\binom{n-k}{\ell}$ tableaux with $\ell$ distinct columns with weights $\Omega^*(q_i)$, $q_i\in\{j_1,j_2,\ldots,j_{n-k}\}$. Since $T_d^A(n-1,n-k)$ contains $\binom{n}{k}$ tableaux, then for each $\Phi\in T_d^A(n-1,n-k)$, the total number of $A$-tableaux $\psi$ corresponding to $\Phi$ is
$$\binom{n}{k}\binom{n-k}{\ell}.$$
However, only $\binom{n}{\ell}$ tableaux in $B_{\Phi}$ with $\ell$ distinct columns of weights other than $r_2$ are distinct. It then follows that every distinct tableau $\psi$ appears
$$\frac{\binom{n}{k}\binom{n-k}{\ell}}{\binom{n}{\ell}}=\binom{n-\ell}{k}$$
times in the collection (cf.\ \cite{CorPJ}). Thus, we consequently obtain
\begin{equation*}
(-1)^{n-k}q^{\binom{n}{2}}w_{m,r,q}(n,k)=\sum_{\ell=0}^{n-k}\binom{n-\ell}{k}r_2^{n-k-\ell}\sum_{\psi\in B_{\ell}}\prod_{c\in\psi}\Omega^{*}(|c|),
\end{equation*}
where $B_{\ell}$ denotes the set of all tableaux $\psi$ having $\ell$ distinct columns whose lengths are in the set $\{0,1,2,\ldots,n-1\}$. Reindexing the double sum yields
\begin{equation}
(-1)^{n-k}q^{\binom{n}{2}}w_{m,r,q}(n,k)=\sum_{j=k}^{n}\binom{j}{k}r_2^{j-k}\sum_{\psi\in B_{n-j}}\prod_{c\in\psi}\Omega^{*}(|c|).\label{tab3.1}
\end{equation}
Since $B_{n-j}=T_d^A(n-1,n-j)$, then
\begin{equation}
\sum_{\psi\in B_{n-j}}\prod_{c\in\psi}\Omega^{*}(|c|)=(-1)^{n-j}q^{\binom{n}{2}}w_{m,r_1,q}(n,j).\label{tab4.1}
\end{equation}
Combining Eqs.~\eqref{tab3.1} and \eqref{tab4.1} gives
\begin{equation}
(-1)^{n-k}q^{\binom{n}{2}}w_{m,r,q}(n,k)=\sum_{j=k}^{n}\binom{j}{k}r_2^{j-k}(-1)^{n-j}q^{\binom{n}{2}}w_{m,r_1,q}(n,j)
\end{equation}
which is equivalent to the desired result in Eq.~\eqref{tab3}. Similarly, if $\phi\in T^A(n-1)$, then substituting $j_i=|\bar{c}|$ in Eq.~\eqref{tab2.2} gives
\begin{equation*}
\omega_A(\phi)=\prod_{i=1}^{n-k}(m[j_i]_q+r),
\end{equation*}
where $j_i\in\{0,1,2,\ldots,k\}$. If for some numbers $r_1$ and $r_2$, $r=r_1+r_2$, then
\begin{eqnarray*}
\omega_A(\phi)&=&\prod_{i=1}^{n-k}\left[(m[j_i]_q+r_1)+r_2\right]\\
&=&\prod_{i=1}^{n-k}(\omega^*(j_i)+r_2),\ \omega^*(j_i)=m[j_i]_q+r_1\\
&=&\sum_{\ell=0}^{n-k}r_2^{n-k-\ell}\sum_{j_1\leq q_1\leq q_2\leq \cdots\leq q_{\ell}\leq j_{n-k}}\prod_{i=1}^{\ell}\omega^{*}(q_i).
\end{eqnarray*}
Suppose $\bar{B}_{\phi}$ is the set of all $A$-tableaux corresponding to $\phi$ such that for each $\zeta\in \bar{B}_{\phi}$, one of the following is true:
\begin{description}
	\item[] $\zeta$ has no column whose weight is $r_2$;
	\item[] $\zeta$ has one column whose weight is $r_2$;
	\item[] $\zeta$ has two columns whose weight are $r_2$;
	\item[] $\vdots$
	\item[] $\zeta$ has $n-k$ columns whose weight are $r_2$.
\end{description}
Then, we may write
\begin{equation*}
\omega_A(\phi)=\sum_{\zeta\in \bar{B}_{\phi}}\omega_A(\zeta).
\end{equation*}
If there are $\ell$ columns in $\zeta$ with weights other than $r_2$, then we have
\begin{eqnarray*}
\omega_A(\zeta)&=&\prod_{\bar{c}\in\zeta}\omega^*(|\bar{c}|)\\
&=&r_2^{n-k\ell}\prod_{i=1}^{\ell}\omega^*(q_i),
\end{eqnarray*}
where $q_1,q_2,q_3,\ldots,q_{\ell}\in\{j_1,j_2,j_3,\ldots,j_{n-k}\}$. It then follows that Eq.~\eqref{tab2.1} may be expressed as
\begin{eqnarray*}
q^{-\binom{k}{2}}W_{m,r,q}(n,k)&=&\sum_{\phi\in T^A(k,n-k)}\omega_A(\phi)\\
&=&\sum_{\phi\in T^A(k,n-k)}\sum_{\zeta\in \bar{B}_{\phi}}\omega_A(\zeta).
\end{eqnarray*}
Like in the previous, for each $\ell$, there correspond $\binom{n-k}{\ell}$ tableaux with $\ell$ columns having weights $\omega^*(q_i)$, $q_i\in\{j_1,j_2,\ldots,j_{n-k}\}$. Since the set $T^A(k,n-k)$ contains $\binom{n}{k}$ tableaux, then for each $\phi\in T^A(k,n-k)$, there are
$$\binom{n}{k}\binom{n-k}{\ell}$$
$A$-tableaux corresponding to $\phi$. But only $\binom{\ell+k}{\ell}$ of these tableaux are distinct. Hence, every distinct tableau $\zeta$ with $\ell$ columns of weights other than $r_2$ appears
$$\frac{\binom{n}{k}\binom{n-k}{\ell}}{\binom{\ell+k}{\ell}}=\binom{n}{\ell+k}$$
times in the collection (cf.\ \cite{tina5}). It implies that
\begin{equation*}
q^{-\binom{k}{2}}W_{m,r,q}(n,k)=\sum_{\ell=0}^{n-k}\binom{n}{\ell+k}r_2^{n-k-\ell}\sum_{\zeta\in \bar{B}_{\ell}}\prod_{\bar{c}\in\zeta}\omega^{*}(|\bar{c}|),
\end{equation*}
where $\bar{B}_{\ell}$ is the set of all tableaux $\zeta$ having $\ell$ columns of weights $\omega^*(j_i)$. Reindexing the sums yield
\begin{equation}
q^{-\binom{k}{2}}W_{m,r,q}(n,k)=\sum_{j=k}^{n}\binom{n}{j}r_2^{n-j}\sum_{\zeta\in \bar{B}_{j-k}}\prod_{\bar{c}\in\zeta}\omega^{*}(|\bar{c}|).\label{tab3.2}
\end{equation}
Since $\bar{B}_{n-j}=T^A(k,n-j)$, then
\begin{equation}
\sum_{\zeta\in \bar{B}_{j-k}}\prod_{\bar{c}\in\zeta}\omega^{*}(|\bar{c}|)=q^{-\binom{k}{2}}W_{m,r_1,q}(j,k).\label{tab4.2}
\end{equation}
Moreover, by Eqs.~\eqref{tab3.2} and \eqref{tab4.2}, we obtain
\begin{equation}
q^{-\binom{k}{2}}W_{m,r,q}(n,k)=\sum_{j=k}^{n}\binom{n}{j}r_2^{n-j}q^{-\binom{k}{2}}W_{m,r_1,q}(j,k)
\end{equation}
which is equivalent to the second desired result.
\end{proof}

Let $r_1=r-1$ and $r_2=1$ in Eqs.~\eqref{tab3} and \eqref{tab4}. Then.
\begin{equation}
w_{m,r,q}(n,k)=\sum_{j=k}^n\binom{j}{k}(-1)^{j-k}w_{m,r-1,q}(n,j)
\end{equation}
and
\begin{equation}
W_{m,r,q}(n,k)=\sum_{j=k}^n\binom{n}{j}W_{m,r-1,q}(j,k).
\end{equation}
These identities were first seen in \cite[Theorem\ 9]{Mah2}. Now, using 
Eq.~\eqref{tab4}, the $(q,r)$-Dowling numbers $D_{m,r,q}(n)$ \cite{Mah2} may be expressed as
\begin{eqnarray*}
D_{m,r,q}(n)&=&\sum_{k=0}^nW_{m,r,q}(n,k)\\
&=&\sum_{k=0}^n\sum_{j=k}^n\binom{n}{j}W_{m,r-1,q}(j,k)\\
&=&\sum_{j=0}^n\binom{n}{j}\sum_{k=0}^jW_{m,r-1,q}(j,k)\\
&=&\sum_{j=0}^n\binom{n}{j}D_{m,r-1,q}(j).
\end{eqnarray*}
Moreover, by applying the binomial inversion formula \cite{Comt}
$$f_n=\sum_{j=0}^n\binom{n}{j}g_j\Longleftrightarrow g_n=\sum_{j=0}^n(-1)^{n-j}\binom{n}{j}f_j$$
to this identity gives
\begin{equation*}
D_{m,r-1,q}(n)=\sum_{j=0}^n(-1)^{n-j}\binom{n}{j}D_{m,r,q}(j).
\end{equation*}
These results are formally stated in the following corollary:
\begin{corollary}
The $(q,r)$-Dowling numbers satisfy the recurrence relations with respect to $r$ given by
\begin{equation}
D_{m,r+1,q}(n)=\sum_{j=0}^n\binom{n}{j}D_{m,r,q}(j)\label{DowlInv1}
\end{equation}
and
\begin{equation}
D_{m,r,q}(n)=\sum_{j=0}^n(-1)^{n-j}\binom{n}{j}D_{m,r+1,q}(j).\label{DowlInv2}
\end{equation}
\end{corollary}

\begin{remark}
When $q\rightarrow 1$ and $m=\beta$, we obtain the following identities by Corcino and Corcino \cite{CorTin2}:
\begin{equation}
G_{n,\beta,r+1}=\sum_{j=0}^n\binom{n}{j}G_{j,\beta,r}
\end{equation}
\begin{equation}
G_{n,\beta,r}=\sum_{j=0}^n(-1)^{n-j}\binom{n}{j}G_{j,\beta,r+1},
\end{equation}
where $G_{n,\beta,r}:=D_{\beta,r,1}(n)$ is the generalized Bell numbers in \cite{Cor,CorTin2}. These identities were used to identify the Hankel transform of $G_{n,\beta,r}$.
\end{remark}

Looking at the previous corollary, we see that the sequence $\big(D_{m,r+1,q}(n)\big)$ is the binomial transform of the sequence $\big(D_{m,r,q}(n)\big)$, for $r=0,1,2,\ldots$. Using ``Layman's Theorem" \cite{Layman}, $\big(D_{m,0,q}(n)\big),\big(D_{m,1,q}(n)\big),\big(D_{m,2,q}(n)\big),\ldots,\big(D_{m,r,q}(n)\big),\ldots$ have the same Hankel transform. This directs our attention to the following open problem:
\begin{problem}
Is it possible to identify the Hankel transform of $D_{m,r,q}(n)$ using a method parallel to what is being done in \cite{CorTin2} for $G_{n,\beta,r}$?
\end{problem}

\subsection{Convolution-type identities}

Recall that for any two sequences $a_n$ and $b_n$, we call the sequence $c_n$ as convolution sequence if
\begin{equation}
c_n=\sum_{k=1}^na_nb_{n-k},\ n=0,1,2,\ldots.
\end{equation}
One of the most famous convolution-type identity is the Vandermonde's formula \cite{Chen,Comt} given by
\begin{equation}
\binom{a+b}{n}=\sum_{k=0}^n\binom{a}{k}\binom{b}{n-k}.
\end{equation}
The following theorem contains convolution-type identities for the $(q,r)$-Whitney numbers of the first kind which will be proved using the combinatorics of $A$-tableaux:

\begin{theorem}
The $(q,r)$-Whitney numbers of the first kind have convolution-type identities given by
\begin{equation}
w_{m,r,q}(p+j,n)=q^{-pj}\sum_{k=0}^nw_{m,r,q}(p,k)w_{\bar{m},\bar{r},q}(j,n-k)\label{convo1}
\end{equation}
and
\begin{equation}
w_{m,r,q}(n+1,j+p+1)=\sum_{k=0}^nq^{k^2-nk-n}w_{m,r,q}(k,p)w_{\bar{m},\bar{r},q}(n-k,j),\label{convo2}
\end{equation}
where $\bar{m}=mq^p$ and $\bar{r}=m[p]_q+r$.
\end{theorem}
\begin{proof}
For $A_1=\{0,1,2,\ldots,p-1\}$ and $A_2=\{p,p+1,p+2,\ldots,p+j-1\}$, let $\Phi_1\in T_d^{A_1}(p-1,p-k)$ and $\Phi_2\in T_d^{A_2}(j-1,j-n+k)$. Note that by joining the columns of the tableaux $\Phi_1$ and $\Phi_2$, we may generate an $A$-tableau $\Phi$ with $p+j-n$ distinct columns whose lengths are in the set $A=\{0,1,2,\ldots,p+j-1\}$. That is, $\Phi\in T_d^{A}(p+j-1,p+j-n)$. Hence,
\begin{equation*}
\sum_{\Phi\in T_d^{A}(p+j-1,p+j-n)}\Omega_A(\Phi)=\sum_{k=0}^n\left\{\sum_{\Phi_1\in T_d^{A_1}(p-1,p-k)}\Omega_{A_1}(\Phi_1)\right\}\left\{\sum_{\Phi_2\in T_d^{A_2}(j-1,j-n+k)}\Omega_{A_2}(\Phi_2)\right\}.
\end{equation*}
Note that in the right-hand side, we get
\begin{eqnarray*}
\sum_{\Phi_2\in T_d^{A_2}(j-1,j-n+k)}\Omega_{A_2}(\Phi_2)&=&\sum_{p\leq g_1<g_2<\cdots<g_{j-n+k}\leq p+j-1}\prod_{i=1}^{j-n+k}(m[g_i]_q+r)\\
&=&\sum_{0\leq g_1<g_2<\cdots<g_{j-n+k}\leq j-1}\prod_{i=1}^{j-n+k}(m[p+g_i]_q+r)\\
&=&\sum_{0\leq g_1<g_2<\cdots<g_{j-n+k}\leq j-1}\prod_{i=1}^{j-n+k}\left(mq^p[g_i]_q+([p]_q+r)\right)\\
&=&(-1)^{j-n+k}q^{\binom{j}{2}}w_{\bar{m},\bar{r},q}(j,n-k),
\end{eqnarray*}
where $\bar{m}=mq^p$ and $\bar{r}=m[p]_q+r$. Also, using Eq.~\eqref{tab1.1},
\begin{equation*}
\sum_{\Phi_1\in T_d^{A_1}(p-1,p-k)}\Omega_{A_1}(\Phi_1)=(-1)^{p-k}q^{\binom{p}{2}}w_{m,r,q}(p,k)
\end{equation*}
and
\begin{equation*}
\sum_{\Phi\in T_d^{A}(p+j-1,p+j-n)}\Omega_A(\Phi)=(-1)^{p+j-n}q^{\binom{p+j}{2}}w_{m,r,q}(p+j,n).
\end{equation*}
Hence, by simplification, we obtain the convolution identity \eqref{convo1}. Similarly, we let $\Phi_1$ be a tableau with $k-p$ columns whose lengths are in $B_1=\{0,1,2,\ldots,k-1\}$ and $\Phi_2$ be a tableau with $n-k-j$ columns whose lengths are in $B_2=\{k+1,k+2,\ldots,n\}$ so that $\Phi\in T_d^{B_1}(k-1,k-p)$ and $\Phi\in T_d^{B_2}(n-k-1,n-k-j)$. Note that we may generate an $A$-tableau $\Phi$ by joining the columns of $\Phi_1$ and $\Phi_2$ whose lengths are in $A=\{0,1,2,\ldots,n\}$. Hence, we have
\begin{equation*}
\sum_{\Phi\in T_d^{A}(n,n-j-p)}\Omega_A(\Phi)=\sum_{k=0}^n\left\{\sum_{\Phi_1\in T_d^{B_1}(k-1,k-p)}\Omega_{B_1}(\Phi_1)\right\}\left\{\sum_{\Phi_2\in T_d^{B_2}(n-k-11,n-k-j)}\Omega_{B_2}(\Phi_2)\right\}.
\end{equation*}
Applying Eq.~\eqref{tab1.1} gives
\begin{equation*}
\sum_{\Phi\in T_d^{A}(n,n-j-p)}\Omega_A(\Phi)=(-1)^{n-j-p}q^{\binom{n+1}{2}}w_{m,r,q}(n+1,j+p+1)
\end{equation*}
and
\begin{equation*}
\sum_{\Phi_1\in T_d^{B_1}(k-1,k-p)}\Omega_{B_1}(\Phi_1)=(-1)^{k-p}q^{\binom{k}{2}}w_{m,r,q}(k,p).
\end{equation*}
Also, in the right-hand side, we get
\begin{eqnarray*}
\sum_{\Phi_2\in T_d^{B_2}(n-k-11,n-k-j)}\Omega_{B_2}(\Phi_2)&=&\sum_{p\leq g_1<g_2<\cdots<g_{n-k-j}\leq p+n-k-1}\prod_{i=1}^{n-k-j}(m[g_i]_q+r)\\
&=&\sum_{0\leq g_1<g_2<\cdots<g_{n-k-j}\leq n-k-1}\prod_{i=1}^{n-k-j}(m[p+g_i]_q+r)\\
&=&\sum_{0\leq g_1<g_2<\cdots<g_{n-k-j}\leq n-k-1}\prod_{i=1}^{n-k-j}\left(mq^p[g_i]_q+([p]_q+r)\right)\\
&=&(-1)^{n-k-j}q^{\binom{n-k}{2}}w_{\bar{m},\bar{r},q}(n-k,j),
\end{eqnarray*}
where $\bar{m}=mq^p$ and $\bar{r}=m[p]_q+r$. This completes the proof.
\end{proof}

The next theorem can be proved similarly.
\begin{theorem}
The $(q,r)$-Whitney numbers of the second kind have convolution-type identities given by
\begin{equation}
W_{m,r,q}(n+1,j+p+1)=\sum_{k=0}^nq^{p+pj+j}W_{m,r,q}(k,p)W_{\hat{m},\hat{r},q}(n-k,j)
\end{equation}
and
\begin{equation}
W_{m,r,q}(p+j,n)=\sum_{k=0}^nq^{nk-k^2}W_{m,r,q}(p,k)W_{\hat{m},\hat{r},q}(j,n-k),
\end{equation}
where $\hat{m}=mq^{p+1}$ and $\hat{r}=m[p+1]_q+r$.
\end{theorem}
As $q\rightarrow 1$, we recover from Theorems 10 and 11 the results recently obtained by Xu and Zhou \cite[Theorems\ 2.1\ and\ 2.4]{Xu}.

\section{On Heine and Euler distributions}

Consider the Poisson distribution
\begin{equation}
f_X(x)=e^{-\lambda}\frac{\lambda^x}{x!},\label{Poiss}
\end{equation}
for $x=0,1,2,\ldots$. 
The factorial moment of a Poisson random variable is readily evaluated, i.e., 
\begin{equation}
E\big[(X)_n\big]=\lambda^n\label{factmom}
\end{equation}
the mean, $E\big[X\big]=\lambda$,  being the special case $n=1$. 
Expanding $x^n$ in terms of falling factorials (using the Stirling numbers of the second kind), we obtain the $n$-th moment of $X$ given by
\begin{equation}
E\big[X^n\big]=B_n(\lambda),\label{expb}
\end{equation}
where $B_n(\lambda)$ are the Bell polynomials. The $q$-analogues of the Poisson distribution introduced by Kemp \cite{Kemp}, and Benkherouf and Bather in \cite{Ben} are given by
\begin{equation}
f_{Y}(y)=e_{q}(-\lambda)q^{\binom{y}{2}}\frac{\lambda^{y}}{[y]_{q}!}, y=0, 1, 2, \ldots\label{heine}
\end{equation}
and
\begin{equation}
f_{Z}(z)=\widehat{e}_{q}(-\lambda)\frac{\lambda^{z}}{[z]_{q}!}, z=0, 1, 2, \ldots.\label{euler}
\end{equation}
These are called Heine and Euler distributions, respectively, where
\begin{equation}
e_{q}(t)=\sum^{\infty}_{k=0}\frac{t^{k}}{[k]_{q}!}\label{expfunc1}
\end{equation}
and
\begin{equation}
\widehat{e}_{q}(t)=\sum^{\infty}_{k=0}q^{\binom{k}{2}}\frac{t^{k}}{[k]_{q}!}.\label{expfunc2}
\end{equation}
In line with this, Charalambides and Papadatos \cite{Cha} obtained the following important results:
\begin{equation}
E\big[[Y]_{r,q}\big]=\frac{q^{\binom{r}{2}}\lambda^{r}}{\prod_{i=1}^r(1+\lambda(1-q)q^{i-1})},\label{qfacth}
\end{equation}
\begin{equation}
E\big[[Z]_{r,q}\big]=\lambda^r,\label{qfacte}
\end{equation}
where $[x]_{r,q}=[x]_q[x-1]_q[x-2]_q\cdots[x-r+1]_q$ is the $q$-falling factorial of $x$ of order $r$. Considering these, we now state the following theorem:
\begin{theorem}
If $Y$ and $Z$ are random variables with Heine and Euler distributions, respectively, and if the mean of $Y$ is $\phi=\frac{\lambda}{1+\lambda(1-q)}$ and the mean of $Z$ is $\lambda$, then
\begin{equation}
E_{\phi}\big[(m[Y]_q+r)^n\big]=\sum_{\ell=0}^n\sum_{i=0}^n(-\lambda)^iq^{-\binom{\ell}{2}-\ell i}\frac{\lambda^{\ell}}{[\ell]_q![i]_q!}\frac{(m[\ell]_q+r)^n}{\prod_{j=1}^{\ell+i}(1+\lambda(1-q)q^{j-1})},\label{qgenD1}
\end{equation}
\begin{equation}
E_{\lambda}\big[(m[Z]_q+r)^n\big]=\widehat{e}_q(-\lambda)\sum_{\ell=0}^n\frac{\lambda^{\ell}}{[\ell]_q!}(m[\ell]_q+r)^n.\label{qgenD}
\end{equation}
\end{theorem}
\begin{proof}
From the defining relation in \eqref{qw1} and the result in \eqref{qfacth},
\begin{equation*}
E_{\lambda}\big[(m[Y]_q+r)^n\big]=\sum_{k=0}^nm^kW_{m,r,q}(n,k)\frac{q^{\binom{k}{2}}\lambda^{k}}{\prod_{j=1}^k(1+\lambda(1-q)q^{j-1})}.
\end{equation*}
Using the explicit formula for the $(q,r)$-Whitney numbers of the second kind \cite[Theorem\ 16]{Mah2} given by
\begin{equation}
W_{m,r,q}(n,k)=\frac{1}{m^{k}[k]_q!}\sum_{\ell=0}^{k}(-1)^{k-\ell}q^{\binom{k-\ell}{2}}\binom{k}{\ell}_q(m[\ell]_q+r)^n,\label{identity3}
\end{equation}
we obtain
\begin{eqnarray*}
E_{\lambda}\big[(m[Z]_q+r)^n\big]&=&\sum_{k=0}^n\left\{\frac{1}{[k]_q!}\sum_{\ell=0}^{k}(-1)^{k-\ell}q^{\binom{k-\ell}{2}}\binom{k}{\ell}_q(m[\ell]_q+r)^n\right\}\\
& &\ \times\frac{q^{\binom{k}{2}}\lambda^{k}}{\prod_{j=1}^k(1+\lambda(1-q)q^{j-1})}\\
&=&\sum_{\ell=0}^n\sum_{k=\ell}^{n}(-1)^{k-\ell}q^{\binom{k-\ell}{2}-\binom{k}{2}}\frac{\lambda^k}{[\ell]_q![k-\ell]_q!}\frac{(m[\ell]_q+r)^n}{\prod_{j=1}^k(1+\lambda(1-q)q^{j-1})}.
\end{eqnarray*}
Reindexing the second sum yields \eqref{qgenD1}. Eq.~\eqref{qgenD} 
may be shown similarly.
\end{proof}

\begin{remark}
When $m=1$ and $r=0$ in the previous theorem, we have
	\begin{equation}
		E_{\phi}\big[[Y]_q^n\big]=\sum_{\ell=0}^n\sum_{i=0}^n(-\lambda)^iq^{-\binom{\ell}{2}-\ell i}\frac{\lambda^{\ell}}{[\ell]_q![i]_q!}\frac{[\ell]_q^n}{\prod_{j=1}^{\ell+i}(1+\lambda(1-q)q^{j-1})},
	\end{equation}
and
	\begin{equation}
		E_{\lambda}\big[[Z]_q^n\big]=\widehat{e}_q(-\lambda)\sum_{\ell=0}^n\frac{\lambda^{\ell}}{[\ell]_q!}[\ell]_q^n\equiv B_{n,q}(\lambda),
	\end{equation}
where $B_{n,q}(\lambda)$ is the $q$-Bell polynomials. On the other hand, if the mean is $\lambda=\frac{x}{m}$,
				$$E_{x/m}\left\{(m[Z]_q+r)^n\right\}=\widehat{e}_q\left(-\frac{x}{m}\right)\sum_{\ell=0}^n\frac{x^{\ell}}{m^{\ell}}\frac{(m[\ell]_q+r)^n}{[\ell]_q!}.$$
This explicit formula is due to Mangontarum and Katriel \cite{Mah2}. Thus
				$$E_{x/m}\big[(m[Z]_q+r)^n\big]=D_{m,r,q}(n,x),$$
where 
\begin{equation}
D_{m,r,q}(n,x)=\sum_{k=0}^{n}W_{m,r,q}(n,k)x^k\label{qDP}
\end{equation}
is the $(q,r)$-Dowling polynomials.
\end{remark}

It is worth mentioning that Mangontarum and Corcino \cite{Mahid} obtained the following pair of $n$-th order generalized factorial moments
\begin{equation}
E_{\lambda}\big[(\beta X+\gamma|\alpha)_{n}\big]=e^{-\lambda}\sum_{i=0}^{\infty}\frac{(i\beta+\gamma|\alpha)_{n}}{i!}\lambda^{i}\label{genfac1}
\end{equation}
\begin{equation}
E_{\lambda}\big[(\alpha X-\gamma|\beta)_{n}\big]=e^{-\lambda}\sum_{i=0}^{\infty}\frac{(i\alpha-\gamma|\beta)_{n}}{i!}\lambda^{i},\label{genfac2}
\end{equation}
where $X$ is a Poisson random variable with mean $\lambda$ and $\alpha$, $\beta$ and $\gamma$ may be real or complex numbers. Here,
\begin{equation}
(t|\alpha)_{n}=t(t-\alpha)(t-2 \alpha)\cdots(t-n\alpha+\alpha),\label{gf}
\end{equation}
with initial conditions $(t|\ \alpha)_{n}=0$ when $n\leq0$ and $(t|\ \alpha)_{0}=1$. 
Notice that \eqref{genfac1} unifies the factorial moment in \eqref{factmom} and the $n$-th moment in \eqref{expb}. More precisely,
\begin{itemize}
	\item when $\beta=1$, $\gamma=0$ and $\alpha=0$,
				$$E_{\lambda}\big[(\beta X+\gamma|\alpha)_{n}\big]=E_{\lambda}\big[X^n\big];$$
	\item when $\beta=1$, $\gamma=0$ and $\alpha=1$,
				$$E_{\lambda}\big[(\beta X+\gamma|\alpha)_{n}\big]=E_{\lambda}\big[(X)_n\big].$$
\end{itemize}
Other known ``Bell-type" and ``Dowling-type" polynomials (see \cite{Cheon,Cor,Mah3,Mah1,Mez,Priv}) can be shown to be particular cases of Eqs.~\eqref{genfac1} and \eqref{genfac2}. Furthermore, Corcino and Mangontarum \cite{CorMah} obtained the generalized $q$-factorial moments 
\begin{equation}
E_{\phi}\big[\left[[\beta Y]_{q}+[\gamma]_{q}|[\alpha]_{q}\right]_{n,q}\big]=\sum_{j=0}^{\infty}\hat{e}_{q^{\beta},j}(-\lambda)\frac{(q^{\beta}\lambda)^j\left[[\beta j]_{q}+[\gamma]_{q}|[\alpha]_{q}\right]_{n,q}}{[j]_{q^{\beta}}!\prod_{i=1}^j\left(1+\lambda(1-q^{\beta})q^{\beta(i-1)}\right)}\label{ngenfacthei1}
\end{equation}
and
\begin{equation}
E_{\lambda}\big[\left[[\beta Z]_{q}+[\gamma]_{q}|[\alpha]_{q}\right]_{n,q}\big]=\hat{e}_{b}(-\lambda)\sum_{j=0}^{\infty}\left[[\beta j]_{q}+[\gamma]_{q}|[\alpha]_{q}\right]_{n,q}\frac{\lambda^{j}}{[j]_{b}!},\label{qgenfacteul1}
\end{equation}
where $Y$ is a random variable with Heine distribution and mean $\phi=\frac{\lambda}{1+\lambda(1-q^{\beta})}$, and $Z$ is a random variable with an Euler distribution and mean $\lambda$. The notations 
\begin{equation}
\left[[\beta Z]_{q}+[\gamma]_{q}|[\alpha]_{q}\right]_{n,q}=\prod_{j=0}^{n-1}\left([\beta t]_q+[\gamma]_q-[\alpha j]_q\right)
\end{equation}
and
\begin{equation}
\hat{e}_{q^{\beta},j}(-\lambda)=\sum_{l=0}^{\infty}\left[\frac{q^{\beta\binom{j}{2}}(-\lambda)^l}{[l]_{q^{\beta}}!\prod_{i=1}^l\left(q^{\beta(i-1)}+\lambda(1-q^{\beta})q^{\beta j}\right)}\right]
\end{equation}
are used.
\eqref{ngenfacthei1} and \eqref{qgenfacteul1} are found to be $q$-analogues of \eqref{genfac1}. By thoroughly investigating \eqref{qgenD}, it is obvious that this result is not generalized by \eqref{ngenfacthei1} and \eqref{qgenfacteul1}.

Privault \cite{Priv} defined an extension of the classical Bell numbers as
\begin{equation*}
e^{ty-\lambda(e^t-t-1)}=\sum_{k=0}^{\infty}B_{n}(y,\lambda)\frac{t^k}{k!}.
\end{equation*}
Moreover, he obtained the following $n$-th moment of a Poisson random variable
\begin{equation}
E_{\lambda}\big[(X+y-\lambda)^{n}\big]=B_{n}(y,-\lambda),\label{GenBell}
\end{equation}
where
\begin{equation}
B_{n}(y,-\lambda)=\sum_{k=0}^n\binom{n}{k}(y-\lambda)^{n-k}\sum_{j=0}^k\sstirling{k}{j}\lambda^j,\label{GenBell2}
\end{equation}
Corcino and Corcino \cite{Cor} showed that the $(r,\beta)$-Bell polynomials satisfy
\begin{equation}
G_{n,\beta,r}(x)=\sum_{k=0}^n\binom{n}{k}r^{n-k}\sum_{j=0}^k\beta^{k-j}\sstirling{k}{j}x^j.\label{GenBell3}
\end{equation}
It then follows that
$$G_{n,1,y-\lambda}(\lambda)=B_{n}(y,-\lambda).$$
The next theorem is analogous to these identities.
\begin{theorem}
The $(q,r)$-Dowling polynomials satisfy the identity
\begin{equation}
D_{m,r,q}(n,x)=\sum_{k=0}^n\binom{n}{k}r^{n-k}\sum_{j=0}^km^{k-j}\sstirling{k}{j}_qx^j.\label{genpriv}
\end{equation}
\end{theorem}

\begin{proof}
Using the binomial theorem, we have
\begin{eqnarray*}
E_{x/m}\big[(m[Z]_q+r)^n\big]&=&\sum_{k=0}^n\binom{n}{k}r^{n-k}m^kE_{x/m}\big[[Z]_q^k\big]\\
&=&\sum_{k=0}^n\binom{n}{k}r^{n-k}m^kB_{n,q}\left(\frac{x}{m}\right)\\
&=&\sum_{k=0}^n\binom{n}{k}r^{n-k}m^{k}\sum_{j=0}^k\sstirling{k}{j}_q\left(\frac{x}{m}\right)^j.
\end{eqnarray*}
The desired result follows from the fact that $E_{x/m}\big[(m[Z]_q+r)^n\big]=D_{m,r,q}(n,x)$.
\end{proof}

\begin{remark}
As $q\rightarrow 1$, we obtain the $(r,\beta)$-Bell polynomial identity in 
Eq.~\eqref{GenBell3}. If the mean is replaced with $\lambda$, then
for an Euler random variable $Z$,
$$E_{\lambda}\big[(m[Z]_q+r)^n\big]=\sum_{k=0}^n\binom{n}{k}r^{n-k}\sum_{j=0}^km^{k}\sstirling{k}{j}_q\lambda^j.$$
As $q\rightarrow 1$, we get \cite[Eq.~34]{Mahid}
$$E_{\lambda}\big[(mX+r)^n\big]=\sum_{k=0}^n\binom{n}{k}r^{n-k}\sum_{j=0}^km^{k}\sstirling{k}{j}\lambda^j.$$
When $m=1$ and $r=y-\lambda$,
\begin{equation}
D_{1,y-\lambda,q}(n,x)=\sum_{k=0}^n\binom{n}{k}(y-\lambda)^{n-k}\sum_{j=0}^k\sstirling{k}{j}_qx^j.\label{genpriv2}
\end{equation}
This is a $q$-analogue of Privault's identity since \eqref{genpriv2} $\rightarrow$ \eqref{GenBell2} as $q\rightarrow 1$.
\end{remark}

\section{Acknowledgment}
The author is thankful to the editor-in-chief and to the referee(s) for
reading carefully the paper and giving helpful comments and
suggestions. Special thanks also to Dr.~Jacob Katriel for sharing his
insights during the early stage of this paper. This research is
supported by the Mindanao State University through the Special Order
No.~26 --- OP, Series of 2016.

\end{document}